\documentclass[10pt]{amsart}
\usepackage[letterpaper, margin=1.5in]{geometry}
\usepackage{graphicx}
\usepackage{amsmath,amssymb,amsthm,mathtools}
\usepackage[all]{xy}
\usepackage[utf8]{inputenc}
\usepackage[hidelinks]{hyperref}
\usepackage{enumitem}

\newtheorem{theorem}{Theorem}[section]
\newtheorem{corollary}[theorem]{Corollary}

\newtheorem{remark}{Remark}[section]
\newtheorem{definition}{Definition}[section]
\newtheorem{lemma}[theorem]{Lemma}
\newtheorem{proposition}[theorem]{Proposition}
\newtheorem{question}[theorem]{Question}
\newtheorem{example}[theorem]{Example}
\newtheorem*{theorem*}{Theorem}
\newtheorem*{question*}{Question}
\begin{document}

\title{Conics associated with Totally Degenerate Curves}
\date{\today}

\author{Qixiao Ma}
\address{Department of Mathematics, Columbia University,}
\email{qxma10@math.columbia.edu}

\begin{abstract}
Let $k$ be a field. Let $X/k$ be a stable curve whose geometric irreducible components are smooth rational curves. Taking Stein factorization of its normalization, we get a conic. We show the conic is non-split in certain cases. As an application, we show for $g\geq3$, the period and index of the universal genus $g$ curve both equal to $2g-2$.

\end{abstract}

\maketitle

\tableofcontents
\section{Introduction}
Let $\Gamma=(V,E)$ be a graph with degree at least $4$ at all the vertices. Let's work over a field $k_0$. Let $X_{\Gamma}/k_{\Gamma}$ be the universal totally degenerate curve with dual graph $\Gamma$. Let $X_\Gamma^\nu$ be the normalization of $X_\Gamma$. Let $K_\Gamma=\Gamma(X_{\Gamma}^\nu,\mathcal{O}_{X_{\Gamma}^\nu})$.
The Stein factorization of $X_\Gamma^{\nu}/k_\Gamma$ produces a conic $C_\Gamma/K_\Gamma$. One may ask:
\begin{question}\label{question}
Does the conic $C_\Gamma$ split (i.e., has a rational point)?
\end{question}
We study this question in general, and show $C_\Gamma$ does not split in certain cases:
\begin{theorem}\label{main0}The conic $C_\Gamma$ does not split, if $\Gamma$ is one of the following graphs:
\begin{enumerate}
\item $V=\{v_i\}_{i\in \mathbb{Z}/(g-1)\mathbb{Z}}$, one edge joining $v_i$ to $v_{i+j}$ for $j\in \{\pm1,\pm2\}$, ($g\geq7$).
\item The complete graph $K_5$.
\item $V=\{v_i\}_{i\in\mathbb{Z}/(g-1)\mathbb{Z}}$, with two edges joining $v_i,v_{i+1}$ ($g\geq4$).
\item $V=\{v_0,v_1\}$, two edges joining $v_0,v_1$ and a loop edge at each $v_i$.
\end{enumerate}
\end{theorem}
We apply the result to reprove the following result:
\begin{theorem}The period and index of the universal genus $g$ curve both equal to $2g-2$.\footnote{This is implied by the strong Franchetta conjecture, see \cite{Schoeer}.}
\end{theorem}

In section \ref{sec2}, we introduce the conic associated with a family of totally degenerate curves. In section \ref{sec3}, we reduce our problem to finding some specific curves over global fields, whose associated conic is non-split. In section \ref{sec4}, we explain the method of constructing totally degenerate curves via clutching, and further reduce our problem to finding certain Brauer class of index two. In section \ref{sec5}, we find the Brauer class by class field theory. In section \ref{sec6}, we finish the proof of Theorem \ref{main0}. In the last sections , we apply the previous results to show the period and index of the universal genus $g$ curve both equal to $2g-2$.
\vspace{1mm}

\textbf{Acknowledgements.} I am very grateful to my advisor Aise Johan de Jong for his invaluable idea, enlightening discussions and unceasing encouragement. Thank Chao Li for helpful discussions on Lemma \ref{keylem}.

\section{Preliminaries}
\subsection{Totally degenerate curves}
Let $k$ be a field, let $k^{\mathrm{sep}}$ be its separable closure. Let $X$ be a one-dimensional scheme defined over $k$.
We say $X$ is a totally degenerate curve, if
$X_{k^{\mathrm{sep}}}$ is a connected union of smooth rational curves, with only nodal singularities, and each irreducible component has at least three nodes. Such curves are Deligne-Mumford stable curves, they are parameterized by nice moduli stacks, see \cite[X.3]{ACGH}.

Recall the dual graph of $X$, is a graph $\Gamma=(V,E)$ whose set of vertices $V$ consists of irreducible components of $X_{k^{\mathrm{sep}}}$, and whose set of edges $E$ consists of nodes in $X_{k^\mathrm{sep}}$: for each node, we assign an edge between the two components that meet at the node. Let's denote the set of edges connected to a vertex $v$ by $E(v)$.

Since $X$ is geometrically connected, we know $p_a(X)=h^1(X_{k^\mathrm{sep}},\mathcal{O}_{X_{k^\mathrm{sep}}})$. A simple argument comparing $X_{k^\mathrm{sep}}$ with its normalization $X_{k^{\mathrm{sep}}}^\nu$ shows $h^1(X_{k^\mathrm{sep}},\mathcal{O}_{X_{k^\mathrm{sep}}})$ coincides with $g(\Gamma):=|E|-|V|+1$.

\subsection{The moduli stacks.}\label{1.2} Let's fix a base field $k_0$. We recall some facts about the moduli stack of totally degenerate nodal curves, see \cite[XII.10]{ACGH} for details.

Let $\Gamma=(V,E)$ be a graph, such that $|E(v)|\geq 3$ for each $v\in V$. Consider the moduli stack $\overline{\mathcal{M}}_{g(\Gamma),k_0}$ of stable genus $g(\Gamma)$ curves over $k_0$.
It has a closed substack $\mathcal{D}_{\Gamma,k_0}$, parameterizing families of totally degenerate nodal curve with dual graph $\Gamma$. Let's omit the subscript $k_0$ when it is clear from context. For any finite set $S$, let $\mathcal{M}_{0,S}$ be the moduli stack of stable genus $0$ curves with $S$-labeled distinct marked points. This is an irreducible smooth $k$ variety.
Let $\mathcal{M}_{\Gamma}=\prod_{v\in V}\mathcal{M}_{0,E(v)}$, there is a natural $\mathrm{Aut}(\Gamma)$-action on $\mathcal{M}_\Gamma$, which permutes the components and reorder the marked points. The stack $\mathcal{D}_\Gamma$ can be naturally identified with quotient stack $[\mathcal{M}_{\Gamma}/\mathrm{Aut}(\Gamma)]$ via clutching morphism $c\colon \mathcal{M}_\Gamma\to\overline{\mathcal{M}}_{g(\Gamma)}$, thus $\mathcal{D}_\Gamma$ is irreducible.

\subsection{Admissible graphs}
Let $\Gamma$ be a graph such that $|E(v)|\geq3$ for each $v\in V$. We say $\Gamma$ is admissible, if the $\mathrm{Aut}(\Gamma)$-action on $\mathcal{M}_\Gamma$ is generically free. We claim $\Gamma$ is admissible, if $|E(v)|\geq 4$ for each $v\in V$. Here is sketch of an argument: Note that for $n\geq 4$, a point in $\mathcal{M}_{0,n}$ represents a smooth rational curve with $n$ marked points. The marked points determine $6\cdot{n\choose 4}$ cross ratios. Pick a $k_0^{\mathrm{sep}}$-point $P$ in $\mathcal{M}_\Gamma=\prod_{v\in V}\mathcal{M}_{0,E(v)}$, such that the cross ratios on all the components are all different. The group $\mathrm{Aut}(\Gamma)$ acts freely on $P$, because the coordinates of $P$ are distinguished by cross ratios. Finally note that being free under a finite group action is an open condition.

Admissible graphs may have vertices of degree $3$. For example, the complete bipartite graph $K_{3,4}$ is admissible: the components with $3$ edges cannot be directly distinguished by cross ratios, but can be distinguished by how the components with $4$ edges are connected to them. The same type of arguments show the notion of admissibility is a property of $\Gamma$ and is independent of $k_0$.

\subsection{The universal curve}\label{k'}
Let $\Gamma$ be a admissible graph. Let $\mathcal{X}_{\Gamma,k_0}\to\mathcal{D}_{\Gamma,k_0}$ be the universal family, restricted from the universal family over $\overline{\mathcal{M}}_{g,k_0}$. The admissibility of $\Gamma$ shows $\mathcal{D}_{\Gamma,k_0}=[\mathcal{M}_{\Gamma,k_0}/\mathrm{Aut}(\Gamma)]$ has an open substack represented by a scheme. Take the generic fiber, we obtain a totally degenerate nodal curve $X_{\Gamma,k_0}$ over $k_{\Gamma,k_0}:=\mathrm{frac}(\mathcal{D}_{\Gamma,k_0})$.

Let $k'$ be the function field of $\mathcal{M}_{\Gamma,k_0}$, then $k'/k_{\Gamma,k_0}$ is a Galois extension with Galois group $\mathrm{Aut}(\Gamma)$, and $X_{k'}$ is a union of $\mathbb{P}^1_{k'}$s with $k'$-rational nodes and dual graph $\Gamma$.

\subsection{Normalization in families}\label{normal} We justify the notion of normalization in families. We use this term only for families of totally degenerate curves with a fixed dual graph.

Let $\Gamma$ be an admissible graph, let $\mathcal{X}_\Gamma\to \mathcal{D}_\Gamma$ be the universal family. Since $\mathcal{D}_\Gamma$ is smooth, by local property of smooth morphism and normalization, we know
the usual normalization $\nu\colon (\mathcal{X}_\Gamma)^\nu\to\mathcal{X}_\Gamma$ is also fiber-wise normalization over $\mathcal{D}_\Gamma$.

Let $Y\to S$ be a family of totally degenerate curves with dual graph $\Gamma$ over an arbitrary scheme $S$. Let $\gamma\colon S\to\mathcal{D}_\Gamma$ be the map to classification stack. We define $Y^\nu$ by $(\mathcal{X}_\Gamma)^\nu\times_{\mathcal{D}_\Gamma,\gamma}S$.
This operation commutes with base change, since fiber products can be composed. 

\section{The Conic}\label{sec2}
Recall in general, a smooth genus-zero curve is called a conic, since the anti-canonical embedding has degree two. In this section, for a family of totally degenerate curves with a fixed dual graph, we define the notion of its associated conic. Let's work over a fixed base field $k_0$.

Let $f\colon X\to S$ be a family of totally degenerate curves with a fixed dual graph $\Gamma$. Let $f^\nu\colon X^\nu\to S$ be the normalization in families defined in section \ref{normal}. Note the geometric components of fibers of $f^\nu$ are all $\mathbb{P}^1$s, so the Stein factorization $\mathrm{St}\colon X^\nu\to\underline{\mathrm{Spec}}((f^\nu)_*\mathcal{O}_{X^\nu})$
is a conic bundle \cite[11.5]{H}.
\begin{definition} We denote conic bundle $\mathrm{St}\colon X^\nu\to\underline{\mathrm{Spec}}((f^\nu)_*\mathcal{O}_{X^\nu})$
by $\mathrm{St}\colon C_f\to S_f$, and call it the conic associated with the family $f$.
\end{definition}
To sum up, we have the following diagram:$$\xymatrix{
C_f\ar@{=}[r]\ar[d]_{\mathrm{St}} &X^\nu\ar[d]_{\mathrm{St}}\ar[rd]^{f^\nu}\ar[r]^{\nu} & X\ar[d]^{f}\\
S_f\ar@{=}[r] &\underline{\mathrm{Spec}}((f^\nu)_*\mathcal{O}_{X^\nu})\ar[r] &S.}$$

The construction of associated conics commutes with arbitrary base change:
\begin{lemma}\label{lemma1} Let $T\to S$ be a morphism, let $f_T\colon X_T\to T$ be the base change of $f$ along $T$, then $S_{(f_T)}\cong S_f\times_ST$ and $C_{(f_T)}\cong C_f\times_ST$.
\end{lemma}
\begin{proof} Since normalization in families commutes with base change, it suffices to show Stein factorization, or $(f^\nu)_*$ commutes with base change. Note that $h^i(X_s^\nu,\mathcal{O}_{X_s})$ is constant for $s\in S$, we conclude by theorems of cohomology and base change, see \cite[II.5]{mum}.
\end{proof}
Here is the specific case that we are interested: Let $\Gamma$ be an admissible graph.
Let $f\colon X_{{\Gamma,k_0}}\to \mathrm{Spec}(k_{\Gamma,k_0})$ be the universal curve. We denote $\Gamma(X_{\Gamma,k_0},\mathcal{O}_{X_{\Gamma,k_0}})$ by $K_{\Gamma,k_0}$ and denote the associated conic $C_f$ by $C_{\Gamma,k_0}$.

In the following study of the associated conic, we mainly focus on the cases when the following two conditions hold:
\begin{itemize}
\item The group $\mathrm{Aut}(\Gamma)$ acts transitively on $V$ and on $E$,
\item For each $v\in V$, $|E(v)|$ is an even number.
\end{itemize}
The first condition ensures the geometrical points of $\mathrm{Spec}(K_{\Gamma,k_0})$ lie in one orbit, so that $K_{\Gamma,k_0}$ is a field. If this condition does not hold, Stein factorization produces a conic over a disconnected space, we can study it over the connected components. If the second condition is not satisfied, the subscheme of nodes on $X_{\Gamma,k_0}$ pulls back to an odd degree zero-cycle on $C_{\Gamma,k_0}=X_{\Gamma,k_0}^\nu$, which forces the conic to split.
\section{Reduction to Global Fields}\label{sec3}
Let $\Gamma=(V,E)$ be a graph with degree at least $4$ at each vertex. Let $k_0$ be field. Let $C_{\Gamma,k_0}/K_{\Gamma,k_0}$ be the associated conic. We show Question \ref{question} can be tested by curves over global fields. First let's reduce the base field $k_0$ to global fields.
\begin{proposition}\label{prop1}If the conic $C_{\Gamma,k_0}$ splits, then there exists a global field $M$ such that the conic $C_{\Gamma,M}$ splits.
\end{proposition}

\begin{proof}
Let $\mathbb{F}$ be the prime field of $k_0$ (i.e., $\mathbb{Q}$ or $\mathbb{F}_p$). Suppose $C_{K,\Gamma}$ splits, then we can pick a section $s\colon\mathrm{Spec}(K_{\Gamma,k_0})\to C_{\Gamma,k_0}$. By Lemma \ref{lemma1}, we know $C_{\Gamma,k_0}=(C_{\Gamma,\mathbb{F}})_{k_0}$. Since $C_{\Gamma,\mathbb{F}}$ is a finite type scheme over $\mathbb{F}$, we know $s$ is defined over some finite type $\mathbb{F}$-subalgebra $A\subset k_0$.
More precisely, the section $s$ is the base change of a section $s_A\colon \mathrm{Spec}(K_{\Gamma,A})\to C_{\Gamma,A}$.

If $\mathrm{char}(k_0)=0$, we take a maximal ideal $\mathfrak{m}$ of $A$, then $M:=A/\mathfrak{m}$ is a finite extension of $\mathbb{Q}$. The section $s_A$ specializes to a section $s_{A/\mathfrak{m}}\colon\mathrm{Spec}(K_{\Gamma,M})\to C_{\Gamma,M}$.
If $\mathrm{char}(k_0)>0$, we take a curve mapping to $\mathrm{Spec}(A)$, replace $A$ by its localization at the generic point of the curve, then do the same argument.
\end{proof}

By the last proposition, we reduce Theorem \ref{main} to the case where the base field $k_0$ is a global field. The next proposition further reduces the problem to finding a totally degenerate curve over certain global field whose associated conic does not split.
\begin{proposition}\label{prop2}If the conic $C_{\Gamma,M}$ splits, then for any field extension $M'/M$, and any totally degenerate curve $g\colon X\to M'$ with dual graph $\Gamma$, the conic $C_g$ splits.
\end{proposition}
\begin{proof}By Lemma \ref{lemma1}, we may assume $M=M'$ and work over $M$ in the rest of the proof.

Let $\mathcal{D}_\Gamma$ be the moduli stack of totally degenerate curves with dual graph $\Gamma$. Let $f\colon \mathcal{X}_\Gamma\to\mathcal{D}_\Gamma$ be the universal curve. We know $\mathcal{D}_{\Gamma}$ is a smooth Deligne-Mumford stack, because it is explicitly presented by $[\mathcal{M}_\Gamma/\mathrm{Aut}(\Gamma)]$.

Let $[X]$ be the closed point in $\mathcal{D}_\Gamma$ corresponding to the curve $X$. Let $d=\mathrm{dim}(\mathcal{D}_\Gamma)$. By \cite[0DR0]{SP}, there exists a flat morphism $\phi\colon\mathrm{Spec}(M[[t_1,\dots,t_d]])\to\mathcal{D}_\Gamma$ centered at $[X]$. Flatness implies that generic point maps to generic point.

Let $f_\phi$ be the family of curves over $\mathrm{Spec}(M[[t_1,\dots,t_d]])$ induced by $\phi$.
Since $C_{\Gamma,M}$ splits, by Lemma \ref{lemma1}, we know the conic $C_{f_\phi}$ also splits. We hope to specialize the sections of $C_{f_\phi}$ to $C_g$. Consider the successive specialization of discrete valuation rings: $$M((t_1,\dots,t_{d-1}))[[t_d]]\rightsquigarrow M((t_1,\dots,t_{d-2}))[[t_{d-1}]]\rightsquigarrow\cdots\rightsquigarrow M[[t_1]].$$
By \cite[9.7]{H}, a section on $C_{f_\phi}$ is inherited by taking closure in each specialization step, which finally yields a section of $C_g$.
\end{proof}

\section{Clutching Curves}\label{sec4}
By the reductions in the last section, it suffices to construct certain totally degenerate curve with non-split associated conic. In this section, we study how to reverse-engineer totally degenerate curves from the conics.

The reverse process of normalization of nodal curves is called clutching, see \cite[X.7]{ACGH}. We explain the process carefully, since we are not working over algebraically closed fields. 
Here is our setup: $$\xymatrix{\mathrm{Spec}(k_5)\ar[r]&\mathrm{Spec}(k_3)\ar[r]^{m\colon1}\ar[d]^{2\colon1}& \mathrm{Spec}(k_2)\ar[d]^{n\colon1} \\
&\mathrm{Spec}(k_4)\ar[r] & \mathrm{Spec}(k_1).}
\eqno{(\dagger)}$$
Let $k_1$ be a field, let $k_2$ be a finite extension of $k_1$. Let $n=[k_2\colon k_1]$. Let $C$ be a conic over $k_2$. Let $i_P\colon P\hookrightarrow C$ be a closed point with residue field $k_3$. Let $m=[k_3\colon k_2]$. Let $k_4$ be an intermediate field between $k_1$ and $k_3$, such that $[k_3\colon k_4]=2$. Let $k_5/k_1$ be the normal closure of $k_3/k_1$. Let's denote the Galois groups $\mathrm{Gal}(k_5/k_i)$ by $G_i$.

The incidence of geometric points in our setup $(\dagger)$ produces a graph $\Gamma_{\dagger}$, with
vertex set $V_\dagger=G_1/G_2$ and edge set $E_\dagger=G_1/G_4$. Two vertices $v,w\in V_\dagger=G_1/G_2$ are connected by an edge $e$ if their pre-images in $G_1/G_3$ map to the same element $e\in E_\dagger=G_1/G_4$. Using our previous notation, there are $n$ vertices in $\Gamma_\dagger$, each vertex has degree $m$.


\begin{example}\label{ex2}Let $\Gamma=(V,E)$ be a graph, such that each vertex has degree at least $4$. Assume $\mathrm{Aut}(\Gamma)$ acts transitively on vertices and on edges. Let $k_5$ be a field with faithful $\mathrm{Aut}(\Gamma)$-action.
Pick an edge $e_0$ with vertices $v_0,v_1$.
Let $G_1=\mathrm{Aut}(\Gamma)$.  Let $G_2=\mathrm{Stab}(\{v_0\})$, $G_3=\mathrm{Stab}(\{v_0\})\cap\mathrm{Stab}(\{e_0\})$ and $G_4=\mathrm{Stab}(\{e_0\})$. Take $k_i=k_5^{G_i}$ in $(\dagger)$, keeping track of the cosets $G_1/G_i$, one verifies that $\Gamma_\dagger=\Gamma$.
\end{example}

Let's explain the clutching construction. By \cite[3.4]{Coproduct}, we may take push-out diagram of the closed immersion $i_P\colon \mathrm{Spec}(k_3)\to C$ and the double cover $h\colon \mathrm{Spec}(k_3)\to\mathrm{Spec}(k_4)$ in the category of schemes over $k_1$. Locally, the co-product $$Y=C\coprod_{i_P,\mathrm{Spec}(k_3),f}\mathrm{Spec}(k_4)$$ is constructed by taking spectrum of fiber product of rings. Note that in this case, the push-out diagram is also cartesian:
\begin{lemma}\label{next}
Let $A, B$ be commutative rings. Let $I$ be an ideal of $A$. Let $B\hookrightarrow A/I$ be a subring. Let $C=A\times_{A/I}B$ be the fibred product of rings. Then $A/I=A\otimes_CB$, where $C\to A$ and $C\to B$ are given by projections.
\end{lemma}
\begin{proof} Consider the map $f\colon A\otimes_CB\to A/I$ given by $(a,b)\mapsto \overline{ab}$, one checks the map $f$ is a surjective ring homomorphism. Consider the map $g\colon A/I\to A\otimes_CB$ given by $\overline{a}\mapsto a\otimes1$. This map is well defined because $i\otimes1=i(1\otimes1)=(1\otimes1)\cdot0=0$ for $i\in I$. It is easy to check $f,g$ are inverse maps to each other, hence we have desired isomorphism.
\end{proof}

\begin{proposition}
The scheme $Y$ is a totally degenerate curve over $k_1$ with dual graph $\Gamma_\dagger$.
\end{proposition}
\begin{proof} By the previous lemma, it suffices give a natural identification of components and nodes of $Y$ with $V_\dagger$ and $E_\dagger$, which preserve the incidence relations. Note that the components are indexed by $\mathrm{Spec}(k_4)$, the nodes are indexed by $\mathrm{Spec}(k_2)$, the nodes pulled back to normalization is induced by $\mathrm{Spec}(k_3)$. We conclude by the Galois correspondence between sub-extensions of $k_5/k_1$ and $G_1$-sets.
\end{proof}

\begin{corollary}\label{prop3} In the setting of Example \ref{ex2}, suppose there exists an index $2$ element in $$\mathrm{Br}(k_5^{G_3}/k_5^{G_2})=\mathrm{Ker}(\mathrm{Br}(k_5^{G_2})\to\mathrm{Br}(k_5^{G_3})),$$ then there exist a totally degenerate curve over $k_1$, with dual graph $\Gamma$ and non-split associated conic.
\end{corollary}
\begin{proof}
Pick a conic $C/k_5^{G_2}$ corresponding to the Brauer class, we have $$C_{k_5^{G_3}}\cong\mathbb{P}^1_{k_5^{G_3}}.$$ Take a closed point $Q$ on $C$ with residue field $k_5^{G_3}$. Take the push-out scheme $Y$ of $Q\to C$ and $Q\to\mathrm{Spec}(k_5^{G_2})$. Then $Y$ is the totally degenerate curve over $k_1$ satisfying required conditions.
\end{proof} 

\section{The Key Lemma}\label{sec5}
\begin{lemma}\label{keylem}Let $N/K$ be a Galois extension of global fields, with Galois group $G$. Let $H$ be a subgroup of $G$. Let $L=N^H$. Suppose there exists an element $g_0$ in $G$ such that the $g_0$-orbit of $G/H$ are all in even size, then $$\mathrm{Br}(L/K)[2]:=\mathrm{Ker}(\mathrm{Br}(K)\to\mathrm{Br}(L))[2]\neq0.$$
\end{lemma}
\begin{proof}Let's denote the set of places in $K$ and $L$ by $S_K$ and $S_L$. By the local-global principle for Brauer groups \cite[VII.10]{CF}, we have the following commutative diagram
$$\xymatrix{
0\ar[r] &\mathrm{Br}(K)\ar[r]\ar[d] &\bigoplus_{v\in S_K}\mathrm{Br}(K_v)\ar[r]^{\oplus\mathrm{inv}_v}_{\sim}\ar[d]&\bigoplus_{v\in S_K}\mathbb{Q/\mathbb{Z}}\ar[r]^-{\sum}\ar[d]^{\rho} &\mathbb{Q}/\mathbb{Z}\ar[d]^{[L\colon K]}\ar[r] &0\\
0\ar[r] &\mathrm{Br}(L)\ar[r] &\bigoplus_{w\in S_L}\mathrm{Br}(L_w)\ar[r]^{\oplus\mathrm{inv}_w}_{\sim}&\bigoplus_{v\in S_L}\mathbb{Q}/\mathbb{Z}\ar[r]^-{\sum} &\mathbb{Q}/\mathbb{Z}\ar[r] &0}$$
If a place $w\in S_L$ lies over place $v\in S_K$, we write $w|v$, and use $f_{(w|v)},e_{(w|v)}$ to denote the residue degree and inertia degree. For a fixed $v\in S_K$, we know $$\sum_{w|v} e_{(w|v)}f_{(w|v)}=[L\colon K].$$ Let's use invariant map to identify Brauer group of local fields with $\mathbb{Q}/\mathbb{Z}$. The restriction map translates to a map $\rho\colon (\mathbb{Q}/\mathbb{Z})^{\oplus S_K}\to(\mathbb{Q}/\mathbb{Z})^{\oplus S_L}$ given by
$$(0,\dots,1_v,\dots,0)\mapsto\sum_{w|v}(0,\dots,e_{(w|v)}f_{(w|v)}\cdot 1_w,\dots,0)$$

We are done if there exist two places $v_1,v_2\in S_K$, such that $f_{(w|v_i)}$ is even for any $w|v_i$, since $(0,\dots,(1/2)_{v_1},\dots,0,\dots,(1/2)_{v_2},\dots,0)$ would be a nonzero class in $\mathrm{Br}(L/K)[2]$.

We show our condition implies the existence of such places.

Let $\mathfrak{p}$ be a prime in $\mathcal{O}_K$ which is unramified in $\mathcal{O}_N$. Let $\mathfrak{Q}\subset \mathcal{O}_N$ be a prime lying over $\mathfrak{p}$, with decomposition group $G_{\mathfrak{Q}}$.
Let $\{\mathfrak{P}_i\}_{i=1}^s$ be all primes in $\mathcal{O}_L$ lying over $\mathfrak{p}$.

Note that have natural bijections: (see \cite[I.9]{Neuk}) $$\xymatrix{\{\mathfrak{P}_i\}_{i=1}^s\ar@{<->}[rrr]^{\mathfrak{P}_i=g\mathfrak{Q}\cap L\mapsto (G_{\mathfrak{Q}})gH}&&& G_{\mathfrak{Q}}\backslash G/H \ar@{<->}[rrr]^{(G_{\mathfrak{Q}})gH\mapsto\mathrm{Orb}_{(G_{\mathfrak{Q}})}(\overline{g})\ \ \ \ }&&&\textrm{$G_{\mathfrak{Q}}$-orbits of}\ G/H}.$$
Under this bijection, the residue degree $f_{(\mathfrak{P}_i|\mathfrak{p})}=|\mathrm{Orb}_{(G_{\mathfrak{Q}})}(\overline{g})|$.
Suppose there exists an element $g_0\in G$ such that all $\langle g_0\rangle$-orbits of $G/H$ are of even size. By Chebotarev's theorem \cite[VII.13.4]{Neuk}, the primes in $\mathcal{O}_N$ whose Frobenius elements are conjugate to $g_0$ has positive Dirichlet density. Hence there are infinitely many primes $\mathfrak{Q}\in\mathcal{O}_N$ whose decomposition group $G_{\mathfrak{Q}}$ is isomorphic to $\langle g_0\rangle$. Pick any two such primes $\mathfrak{Q}_1,\mathfrak{Q}_2$ over different primes in $\mathcal{O}_K$, then $\mathfrak{Q}_i\cap \mathcal{O}_K$ give the desired places of $K$.
\end{proof}

\begin{remark}\label{rm} The lemma holds if $G\backslash\bigcup_{g\in G}g^{-1}Hg$ contains an element $g_0$ whose order is a power of $2$. Here is the reason: For such $g_0$, the only possibility to have an odd size $\langle g_0\rangle$-orbit in $G/H$ is an orbit of length $1$. This means $g_0gH=gH$, or equivalently $g_0\in gHg^{-1}$ for some $g\in G$. Specifically, the condition holds if $G$ has an order $2^s$ element but $H$ does not.
\end{remark}
\begin{remark}\label{keyrem} Let $G=S_4, H=S_3$. Let $K$ be a global field with faithful $S_4$-action, then $\mathrm{Br}(K^{S_3}/\mathrm{K^{S_4}})[2]\neq0$, since $S_4$ has an order $4$ element but $S_3$ does not.
\end{remark}
\begin{remark} In general, for an even degree extension $L/K$, it is possible that $\mathrm{Br}(L/K)[2]=0$.
Let $K$ be some totally imaginary field, so archimedean places give no contribution. Let $N/K$ be an unramified extension with Galois group $A_4$, let $H$ be an order $2$ subgroup in $G$, let $L=N^H$. The local map for $\mathrm{Br}(K)\to\mathrm{Br}(L)$ are multiplication by the size of orbits. At each local place, the Frobenius element has order $1,2$ or $3$. If the Frobenius element has order $1,3$, then there is odd size orbit; in case the Frobenius element has order $2$, there still exists size $1$ orbit, since all the order $2$ elements in $A_4$ are conjugate, see Remark \ref{rm}.
\end{remark} 

\section{The Proof}\label{sec6}
We assemble the previous discussions and finish the proof.

\begin{proof}[Proof of Theorem \ref{main0}] Note that for any global field $M$, there exists a field extension $\widetilde{M}$, with faithful $\mathrm{Aut}(\Gamma)$ action, such that $M\subset \widetilde{M}^{\mathrm{Aut}(\Gamma)}$. For example, we may take the splitting field of a degree $n:=|\mathrm{Aut}(\Gamma)|$ polynomials over $M$ with general coefficients\footnote{This is a $S_n$-Galois extension, by Hilbert's irreducibility theorem, see \cite[9.4.2]{Lang}.}, then take invariant subfield of $\mathrm{Aut}(\Gamma)\subset S_{n}$.

One verifies Lemma \ref{keylem} holds for all cases in Theorem \ref{main0}: Case (2) follows from Remark \ref{keyrem}. In case (1) $G_2/G_3=\mathbb{Z}/2\mathbb{Z}$, in case (3) and (4), we have $G_2/G_3=(\mathbb{Z}/2\mathbb{Z})^{\oplus2}$, one can always find an order $2$ element in $G_2$ such whose orbits in $G_2/G_3$ are all of even size\footnote{Take the inverse image of $(\overline{0},\overline{1})$ in $G_3$, both of its orbits has size two.}. Thus there exists a nonzero class $\alpha\in \mathrm{Br}(\widetilde{M}^{S_3}/\widetilde{M}^{S_4})[2].$
Note that for Brauer classes over global fields, period equals index \cite[18.6]{pierce}, we know the class $\alpha$ has index $2$.

By Corollary \ref{prop3}, the index $2$ class $\alpha$ produces a totally degenerate curve $X$ over $\widetilde{M}^{S_5}$ with dual graph $\Gamma=K_5$, such that the associated conic is non-split.

By Proposition \ref{prop2}, the existence of $X$ implies that the conic $C_{\Gamma,M}$ is non-split. Up to here, our argument works for any global field $M$.

Finally, by Proposition \ref{prop1}, we know the conic $C_{\Gamma,k_0}$ is non-split.
\end{proof}

\section{Non-triviality of Picard torsors}\label{nontri}
Let $k$ be a field. Let $X$ be a totally degenerate nodal curve defined over $k$. Let $\mathrm{Pic}^{\mathbf{1}}_{X/k}$ be the component of the Picard scheme that parameterizes line bundles with degree one on each geometrical irreducible component. The scheme $\mathrm{Pic}^{\mathbf{1}}_{X/k}$ is a torsor of the Picard torus $\mathrm{Pic}^0_{X/k}$. We show that the torsor is non-trivial in certain cases.

\begin{lemma}\label{pictor} Let $C$ be a non-split conic defined over $k$. Let $P$ be a separable closed point of degree $2$ on $C$. Let $D$ be the cone on $P$, we identify $P$ with its section at infinity. Let $X=C\coprod_{P}D$, it is a genus $1$ curve with three rational components. Then $\mathrm{Pic}^{\mathbf{1}}_{X/k}$ does not have a rational point.
\end{lemma}
\begin{proof}We construct an isomorphism $\mathcal{A}\colon C\backslash\{P\}\to \mathrm{Pic}^1_{X/k}$, then conclude from our assumption on non-splitness of $C$.  Take a pair of Galois conjugate points $d,d'$ on $D\backslash\{P\}$. Then $\mathcal{O}_D(d+d')$ is a line bundle defined over $k$, with degree one on each component of $D$. Up to scaling, it has a canonical section $s_{d+d'}$ vanishing at $d,d'$.

For any geometric point $c\in C_{\overline{k}}\cong\mathbb{P}^1_{\overline{k}}$, up to scaling, the line bundle $\mathcal{O}_{\mathbb{P}^1_{\overline{k}}}(1)$ has a canonical section $s_c$ vanishing at $c$. Let $x,x'$ be the geometrical points of $P$. We assign $c$ with the line bundle $\mathcal{A}(c)$ on $X_{\overline{k}}$, obtained from $\mathcal{O}_{\mathbb{P}^1_{\overline{k}}}(1)$ and $\mathcal{O}_D(d+d')$ by identifying the section of $s_c$ with $s_{d+d'}$ at $x$ and $x'$. Then clearly $\mathcal{A}$ is an isomorphism and descends to $k$.
\end{proof}

\begin{proposition}Let $\Gamma$ be the graph as in Theorem \ref{main0}. Let's work over a field $k_0$. Let $X_\Gamma/k_\Gamma$ be the universal genus $g$ curve with dual graph $\Gamma$. Then the $\mathrm{Pic}^0_{X_\Gamma/k_\Gamma}$-torsor $\mathrm{Pic}^{\mathbf{1}}_{X_\Gamma/k_\Gamma}$ does not have a rational point.
\end{proposition}
\begin{proof} We prove the proposition when the graph $\Gamma$ is the circulant graph as case (1) in Theorem \ref{main0}. The other cases are similar.

Note that $\mathrm{Aut}(\Gamma)=D_{2g-2}$, the dihedral group with order $2g-2$, so the curve $X_\Gamma$ is split by some Galois extension $k'/k_\Gamma$ with Galois group $D_{2g-2}$. Let $\tau\in D_{2g-2}$ be a reflection which fixes vertex $v_0$. Since $\mathrm{Aut}(\Gamma)$ acts transitively on the geometric components, we know the geometrical components of $X_\Gamma$ are indexed by $k^{\prime\mathrm{Stab}(v_0)}=k^{\prime\tau}$. Thus the geometrical components of the base change $(X_{\Gamma})_{k^{\prime\tau}}$ are indexed by the separable extension $(k^{\prime\tau}\otimes_{k'} k^{\prime\tau})/k^{\prime\tau}$, or equivalently the $\tau$-orbits of $D_{2g-2}$. We show the base change $(\mathrm{Pic}^{\mathbf{1}}_{X_\Gamma/k_\Gamma})_{k^{\prime\tau}}=\mathrm{Pic}^{\mathbf{1}}_{(X_{\Gamma})_{k^{\prime\tau}}/k^{\prime\tau}}$ does not have a rational point.

Let $L_i:=[v_{i-1}v_iv_{i+1}]$ and $L_\infty:=[v_0\dots v_{g-2}]$ be the loops in $\Gamma$. We have $\mathrm{H}^1(\Gamma,\mathbb{Z})=\langle L_\infty\rangle\oplus \left(\bigoplus_{i\neq0,\frac{g-1}{2}} \langle L_i\rangle\right)\oplus\left(\bigoplus_{i\in\{0,\frac{g-1}{2}\}} \langle L_i\rangle\right)$, thus the Picard torus $\mathrm{Pic}^0_{(X_\Gamma)_{k^{\prime\tau}}/k^{\prime\tau}}$, whose character lattice is $\mathrm{H}^1(\Gamma,\mathbb{Z})$, splits accordingly into product of tori $T_\infty\times T_{\mathrm{conj}}\times T_{\mathrm{inv}}$.
Hence the torsor $\mathrm{Pic}_{(X_\Gamma)_{k^{\prime\tau}}/k^{\prime\tau}}^{\mathbf{1}}$ decomposes accordingly into $P_\infty\times P_{\mathrm{conj}}\times P_{\mathrm{inv}}$, and it suffices to show the torsor $P_{\mathrm{inv}}$ does not have a rational point. Note that the conic $X_\Gamma^\nu/k^{\prime\tau}$ is non-split by Theorem \ref{main0}, we conclude from Lemma \ref{pictor}. \end{proof}



\section{Period and index of the universal curve}
Let's work over a field $k_0$. Let $X/k$ be the universal genus $g$ curve. We show:
\begin{theorem}\label{main}The torsor $[\mathrm{Pic}^1_{X/k}]\in \mathrm{H}^1(k,\mathrm{Pic}^0_{X/k})$ has order $2g-2$.
\end{theorem}
\begin{proof} Consider the Hochschild-Serre spectral sequence for the base change $(\mathrm{Pic}^0_{X/k})_{k^{\mathrm{sep}}}\to\mathrm{Pic}^0_{X/k}$. Let $\mathrm{Br}_{\mathrm{alg}}(\mathrm{Pic}^0_{X/k}):=\mathrm{Ker}(\mathrm{Br}(\mathrm{Pic}^0_{X/k})\to\mathrm{Br}((\mathrm{Pic}^0_{X_{k}/k})_{k^{\mathrm{sep}}})$. The Brauer group can be filtered into: $0\subset\mathrm{Br}(k)\subset\mathrm{Br}_{\mathrm{alg}}(\mathrm{Pic}^0_{X/k})\subset\mathrm{Br}(\mathrm{Pic}^0_{X/k}).$
Since $\mathrm{Pic}^0_{X/k}$ has the identity section, we know $\mathrm{Br}_{\mathrm{alg}}(\mathrm{Pic}^0_{X/k})$ factors into $\mathrm{Br}(k)\oplus\mathrm{Ker}(\mathrm{d}_2^{1,1})$, where $\mathrm{d}_{2}^{1,1}$ is the transgression map $\mathrm{H}^1(k,\mathrm{Pic}_{\mathrm{Pic}^0_{X/k}/k})\to\mathrm{H}^3(k,\mathbb{G}_m)$.
One can explicitly check that the class of the torsor $[\mathrm{Pic}_{X/k}^1]\in\mathrm{H}^1(k,\mathrm{Pic}^0_{X/k})=\mathrm{H}^1(k,\mathrm{Pic}^0_{\mathrm{Pic}^0_{X/k}/k})$ maps to the obstruction class $\alpha\in\mathrm{Br}(\mathrm{Pic}^0_{X/k})$. One can show $\mathrm{per}(\alpha)=g-1$ \cite[4.3.4]{Qixiao}, thus the period of $[\mathrm{Pic}^1_{X/k}]$ is a multiple of $g-1$. It suffice to show the torsor $\mathrm{Pic}^{g-1}_{X/k}$ does not have a $k$-rational point.

By the discussion in Section \ref{nontri}, we may take a graph $\Gamma$ with genus $g$, such that $\mathrm{Pic}^\mathbf{1}_{X_\Gamma/k_\Gamma}$ does not have a rational point. Let $R_1=k_\Gamma[[t_1]]$. By \cite[B.2]{Conrad}, we may find a regular surface $\mathcal{X}/R_1$, such that the special fiber is $X_\Gamma/k_\Gamma$ and generic fiber $X_1/k_1$ is smooth. We show the component $\mathrm{Pic}^{g-1}_{X_1/k_1}$ does not have rational points.

Let $R_1^{\mathrm{sh}}$ be the strict henselization of $R_1$. Note that a rational point $Q_1$ on $\mathrm{Pic}^{g-1}_{X_1/k_1}$ represents a line bundle $L$ on the generic fiber of $(\mathcal{X})_{R_1^{\mathrm{sh}}}$. Since $\mathcal{X}$ is regular, so is $(\mathcal{X})_{R_1^{\mathrm{sh}}}$. The line bundle $L$ canonically extends to $(\mathcal{X})_{R_1^{\mathrm{sh}}}$, by taking closure of the corresponding Weil divisor. The line bundle on the special fiber yields a rational point $Q_0$ on $\mathrm{Pic}_{X_\Gamma/k_\Gamma}$. Let $k'/k_\Gamma$ be the splitting field of $X_\Gamma$, then $Q_0$ represents a line bundle on $(X_\Gamma)_{k'}$. Note that $Q_0$ is $\mathrm{Gal}(k'/k_\Gamma)$-invariant, and that all the $(g-1)$ components of $(X_\Gamma)_{k'}$ are Galois conjugate, thus $Q_0$ lands in $\mathrm{Pic}^\mathbf{1}_{X_\Gamma/k_\Gamma}$. This contradicts our assumption on non-triviality of $\mathrm{Pic}^{\mathbf{1}}_{X_\Gamma/k_\Gamma}.$

Apply \cite[0DR0]{SP} to the stack $\overline{\mathcal{M}}_g$, we pick a flat families of curves over $k_1[[t_1,\dots,t_m]]$, such that the special fiber is $X_1/k_1$, the generic fiber $X_m/k_m$ is a base change of $X/k$. We slice the family into families of curves over discrete valuation rings $R_i=k_\Gamma((t_1,\dots,t_{i-1}))[[t_i]]$ for $i=2,\dots m$. By properness, any rational point on $\mathrm{Pic}_{X_m/k_m}^{g-1}$ specializes by step by step to rational points $\mathrm{Pic}^{g-1}_{X_1/k_1}$. Our previous discussion show there does not exist rational points on $\mathrm{Pic}^{g-1}_{X_m/k_m}=(\mathrm{Pic}^{g-1}_{X/k})_{k_m}$, hence $\mathrm{Pic}^{g-1}_{X/k}$ does not have rational points.
\end{proof}
\begin{theorem} The index $\mathrm{ind}(X):=\mathrm{g.c.d}\{\mathrm{deg}(Z)|\textrm{closed\ subscheme}\ Z{\varsubsetneq} X\}$ equals $2g-2$.
\end{theorem}
\begin{proof} Note that the canonical divisor always provide a degree $2g-2$ zero cycle on $X$. On the other hand, by \cite{LT}, the period of $X$, or the order of $[\mathrm{Pic}^1_{X/k}]$, always divides $\mathrm{ind}(X)$. We conclude from Theorem \ref{main}.

An alternative way is to observe that index drops after specialization. Pick the graphs as in Theorem \ref{main0}, it suffices show that closed points on $X_\Gamma$ has degree a multiple of $2g-2$. For any closed subscheme $Z\subset X_\Gamma$, its irreducible components either lie in the smooth locus of $X_\Gamma$ or in the singular locus of $X_\Gamma$.
We conclude from the non-splitness of the conic $C_\Gamma$, or from counting the degree of singular locus.

\end{proof}

\bibliographystyle{alpha}
\bibliography{references}

\end{document}